\documentclass[12pt]{article} 
\usepackage{amsfonts,amsmath,latexsym,amssymb,mathrsfs,amsthm,comment}
\usepackage{slashbox}
\usepackage{caption}

\let\OLDthebibliography\thebibliography
\renewcommand\thebibliography[1]{
  \OLDthebibliography{#1}
  \setlength{\parskip}{3pt}
  \setlength{\itemsep}{0pt plus 0.3ex}
}

\def\numberlikeadb{\global\def\theequation{\thesection.\arabic{equation}}}
\numberlikeadb
\newtheorem{theorem}{Theorem}[section]

\newtheorem{corollary}[theorem]{Corollary}

\newtheorem{remark}[theorem]{Remark}

\usepackage{lscape}
\usepackage{caption}
\usepackage{multirow}
\begin{document}

\title{Inequalities for integrals of the modified Struve function of the first kind II}
\author{Robert E. Gaunt\footnote{School of Mathematics, The University of Manchester, Manchester M13 9PL, UK}}

\date{\today} 
\maketitle

\vspace{-5mm}

\begin{abstract}Simple inequalities are established for integrals of the type $\int_0^x \mathrm{e}^{-\gamma t} t^{-\nu} \mathbf{L}_\nu(t)\,\mathrm{d}t$, where $x>0$, $0\leq\gamma<1$, $\nu>-\frac{3}{2}$ and $\mathbf{L}_{\nu}(x)$ is the modified Struve function of the first kind. In most cases, these inequalities are tight in certain limits.  As a consequence we deduce a tight double inequality, involving the modified Struve function $\mathbf{L}_{\nu}(x)$, for a generalized hypergeometric function.  
\end{abstract}



\noindent{{{\bf{AMS 2010 Subject Classification:}}} Primary 33C10; 26D15

\section{Introduction}\label{intro}

In a series of recent papers \cite{gaunt ineq1, gaunt ineq3, gaunt ineq6}, simple lower and upper bounds, involving the modified Bessel function of the first kind $I_\nu(x)$, were obtained for the integrals
\begin{equation}\label{intbes}\int_0^x \mathrm{e}^{-\gamma t} t^{\pm\nu} I_\nu(t)\,\mathrm{d}t,
\end{equation}
where $x>0$, $0\leq\gamma<1$ and $\nu>-\frac{1}{2}$.  For $\gamma\not=0$ there does not exist simple closed form expressions for these integrals. The inequalities of \cite{gaunt ineq1,gaunt ineq3} were needed in the development of Stein's method \cite{stein,chen,np12} for variance-gamma approximation \cite{eichelsbacher, gaunt vg, gaunt vg2}. Although, as they are simple and surprisingly accurate the inequalities may also prove useful in other problems involving modified Bessel functions; see for example, \cite{baricz3} in which inequalities for modified Bessel functions of the first kind were used to obtain lower and upper bounds for integrals involving modified Bessel functions of the first kind.

The  modified Struve function of the first kind, defined for $x\in\mathbb{R}$ and $\nu\in\mathbb{R}$ by
\begin{equation*}\mathbf{L}_\nu(x)=\sum_{k=0}^\infty \frac{\big(\frac{1}{2}x\big)^{\nu+2k+1}}{\Gamma(k+\frac{3}{2})\Gamma(k+\nu+\frac{3}{2})},
\end{equation*}
is closely related to the modified Bessel function $I_\nu(x)$, and either shares or has a close analogue to the properties of $I_\nu(x)$ that were used by \cite{gaunt ineq1,gaunt ineq3,gaunt ineq6} to obtain inequalities for the integrals in (\ref{intbes}).  The function $\mathbf{L}_\nu(x)$ is itself a widely used special function; see a standard reference, such as \cite{olver}, for its basic properties.  It has numerous applications in the applied sciences, including leakage inductance in transformer windings \cite{hw94}, perturbation approximations of lee waves in a stratified flow \cite{mh69}, scattering of plane waves by soft obstacles \cite{s84}; see \cite{bp13} for a list of further application areas. 

It is therefore a natural problem to ask for simple inequalities, involving the modified Struve function of the first kind, for the integrals
\begin{equation}\label{intstruve}\int_0^x \mathrm{e}^{-\gamma t} t^{\nu} \mathbf{L}_\nu(t)\,\mathrm{d}t, \qquad \int_0^x \mathrm{e}^{-\gamma t} t^{-\nu} \mathbf{L}_\nu(t)\,\mathrm{d}t
\end{equation}
where $x>0$, $0\leq\gamma<1$ and $\nu>-\frac{3}{2}$.  

When $\gamma=0$ both integrals in (\ref{intstruve}) can be evaluated exactly, because the modified Struve function $\mathbf{L}_{\nu}(x)$ can be represented as a generalized hypergeometric function.  To see this, recall that the generalized hypergeometric function (see \cite{olver} for this definition and further properties) is defined by
\begin{equation*}{}_pF_q\big(a_1,\ldots,a_p;b_1,\ldots,b_q;x\big)=\sum_{k=0}^\infty \frac{(a_1)_k\cdots(a_p)_k}{(b_1)_k\cdots(b_q)_k}\frac{x^k}{k!},
\end{equation*}
and the Pochhammer symbol is given by $(a)_0=1$ and $(a)_k=a(a+1)(a+2)\cdots(a+k-1)$, $k\geq1$.  Then, for $-\nu-\frac{3}{2}\notin\mathbb{N}$, we have the representation
\begin{equation*}\mathbf{L}_\nu(x)=\frac{x^{\nu+1}}{\sqrt{\pi}2^\nu\Gamma(\nu+\frac{3}{2})} {}_1F_2\bigg(1;\frac{3}{2},\nu+\frac{3}{2};\frac{x^2}{4}\bigg)
\end{equation*}
(see also \cite{bp13} for other representations in terms of the generalized hypergeometric function). A straightforward calculation then yields
\begin{equation}\label{besint6}\int_0^x \frac{ \mathbf{L}_\nu(t)}{t^\nu}\,\mathrm{d}t=\frac{x^{2}}{\sqrt{\pi}2^{\nu+1}\Gamma(\nu+\frac{3}{2})}{}_2F_3\bigg(1,1;\frac{3}{2},2,\nu+\frac{3}{2};\frac{x^2}{4}\bigg),
\end{equation}
with a similar formula available for $\int_0^x t^{\nu} \mathbf{L}_\nu(t)\,\mathrm{d}t$.  When $\gamma\not=0$, there does, however, not exist a closed form formula for the integrals in (\ref{intstruve}).  Moreover, even when $\gamma=0$ the first integral is given in terms of the generalized hypergeometric function.  This provides the motivation for establishing simple bounds, involving the modified Struve function $\mathbf{L}_\nu(x)$, for these integrals.


Inequalities were established by \cite{gaunt ineq4} for the first integral in (\ref{intstruve}) by adapting the techniques used by \cite{gaunt ineq1,gaunt ineq3} to bound the related integral involving the modfied Bessel function $I_\nu(x)$.   In this note, we obtain lower and upper bounds for the second integral in (\ref{intstruve}). We proceed in a similar manner to \cite{gaunt ineq4} by adapting the methods used \cite{gaunt ineq6} to bound related integrals involving $I_\nu(x)$, and the inequalities obtained in this note take a similar form to those obtaned by \cite{gaunt ineq6}.  As already noted, the reason for this similarity is because many of the properties of $I_\nu(x)$ that were exploited in the proofs of  \cite{gaunt ineq1,gaunt ineq3,gaunt ineq6} are shared by $\mathbf{L}_\nu(x)$, which we now list.  All these formulas can be found in \cite{olver}, except for the inequality which is given in \cite{bp14}.  Further inequalities for $\mathbf{L}_\nu(x)$ can be found in \cite{bp14,bps17,gaunt ineq5,jn98}, some of which improve on the inequality of \cite{bp14}.

For positive values of $x$ the function $\mathbf{L}_{\nu}(x)$ is positive for $\nu>-\frac{3}{2}$ . The function $\mathbf{L}_{\nu}(x)$ satisfies the recurrence relation and differentiation formula
\begin{align}\label{Iidentity}\mathbf{L}_{\nu -1} (x)- \mathbf{L}_{\nu +1} (x) &= \frac{2\nu}{x} \mathbf{L}_{\nu} (x)+\frac{\big(\frac{1}{2}x\big)^\nu}{\sqrt{\pi}\Gamma(\nu+\frac{3}{2})}, \\
\label{diffone}\frac{\mathrm{d}}{\mathrm{d}x} \bigg(\frac{\mathbf{L}_{\nu} (x)}{x^\nu} \bigg) &= \frac{\mathbf{L}_{\nu +1} (x)}{x^\nu}+\frac{2^{-\nu}}{\sqrt{\pi}\Gamma(\nu+\frac{3}{2})},
\end{align}
and has the following asymptotic properties:
\begin{align}\label{Itend0}\mathbf{L}_{\nu}(x)&\sim \frac{2}{\sqrt{\pi}\Gamma(\nu+\frac{3}{2})}\bigg(\frac{x}{2}\bigg)^{\nu+1}, \quad x \downarrow 0, \: \nu>-\tfrac{3}{2}, \\
\label{Itendinfinity}\mathbf{L}_{\nu}(x)&\sim \frac{\mathrm{e}^{x}}{\sqrt{2\pi x}}, \quad x \rightarrow\infty, \: \nu\in\mathbb{R}.
\end{align}
Let $x > 0$. Then 
\begin{equation}\label{Imon}\mathbf{L}_{\nu} (x) < \mathbf{L}_{\nu - 1} (x), \quad \nu \geq \tfrac{1}{2}.  
\end{equation} 

We end this introduction by noting that \cite{gaunt ineq6} also derived lower and upper bounds for the integral $\int_x^\infty \mathrm{e}^{\gamma t} t^{-\nu} K_\nu(t)\,\mathrm{d}t$, where $x>0$, $\nu>-\frac{1}{2}$, $0\leq\gamma<1$ and $K_\nu(x)$ is a modified Bessel function of the second kind.  Analogously to the problem studied in this note it is natural to ask for bounds for the integral $\int_x^\infty \mathrm{e}^{\gamma t} t^{-\nu} \mathbf{M}_\nu(t)\,\mathrm{d}t$, where $\mathbf{M}_\nu(x)=\mathbf{L}_\nu(x)-I_\nu(x)$ is the modified Struve function of the second kind.  However, the inequalities of \cite{gaunt ineq6} do not have a natural analogue for $\mathbf{M}_\nu(x)$; a discussion as to why this is the case is given in the Introduction of \cite{gaunt ineq4}.

\section{Inequalities for integrals of the modified Struve function of the first kind}\label{sec2}

The following theorem complements the inequalities for the integral $\int_0^x \mathrm{e}^{-\gamma t}t^\nu \mathbf{L}_\nu(t)\,\mathrm{d}t$ that are given in Theorem 2.1 of \cite{gaunt ineq4}.  The inequalities are natural analogues of the inequalities obtained in Theorem 2.5 of \cite{gaunt ineq6} for the related integrals involving the modified Bessel function $I_\nu(x)$.  Before stating the theorem, we introduce the notation
\begin{align*}a_{\nu,n}&=\frac{2\nu+n+1}{\sqrt{\pi}2^{\nu+n+2}(n+2)(\nu+n+1)\Gamma(\nu+n+\frac{5}{2})}, \\
b_{\nu,n}&=\frac{(2\nu+n+1)(2\nu+n+3)}{\sqrt{\pi}2^{\nu+n+4}(n+1)(n+4)(\nu+n+3)\Gamma(\nu+n+\frac{9}{2})}, \\
c_{\nu,n}&=\frac{2\nu+n+1}{\sqrt{\pi}2^{\nu+n+1}(n+1)(n+2)\Gamma(\nu+n+\frac{5}{2})}.
\end{align*}

\begin{theorem}\label{tiger1}Let $0<\gamma<1$ and $n>-1$. Then, for all $x>0$,
\begin{align}\label{bi1}\int_0^x \frac{\mathbf{L}_\nu(t)}{t^\nu}\,\mathrm{d}t&>\frac{\mathbf{L}_\nu(x)}{x^\nu}-\frac{x}{\sqrt{\pi}2^\nu\Gamma(\nu+\frac{3}{2})}, \quad \nu>-\tfrac{3}{2}, \\
\label{bi2}\int_0^x \frac{\mathbf{L}_{\nu+n}(t)}{t^\nu}\,\mathrm{d}t&>\frac{\mathbf{L}_{\nu+n+1}(x)}{x^\nu}-a_{\nu,n}x^{n+2}, \quad \nu>-\tfrac{1}{2}(n+1), \\
\label{bi3}\int_0^x \frac{\mathbf{L}_{\nu+n}(t)}{t^\nu}\,\mathrm{d}t&<\frac{2(\nu+n+1)}{n+1}\frac{\mathbf{L}_{\nu+n+1}(x)}{x^\nu}-\frac{2\nu+n+1}{n+1}\frac{\mathbf{L}_{\nu+n+3}(x)}{x^\nu}\nonumber \\
&\quad+b_{\nu,n}x^{n+4}-c_{\nu,n}x^{n+2}, \quad \nu>-\tfrac{1}{2}(n+1), \\
\label{bi4}\int_0^x \mathrm{e}^{-\gamma t}\frac{\mathbf{L}_\nu(t)}{t^\nu}\,\mathrm{d}t&>\frac{1}{1-\gamma}\bigg(\mathrm{e}^{-\gamma x}\int_0^x\frac{\mathbf{L}_\nu(t)}{t^\nu}\,\mathrm{d}t-\frac{1-(1+\gamma x)\mathrm{e}^{-\gamma x}}{\sqrt{\pi}\gamma2^\nu\Gamma(\nu+\frac{3}{2})}\bigg), \quad \nu>-\tfrac{3}{2}, \\
\label{bi5}\int_0^x \mathrm{e}^{-\gamma t}\frac{\mathbf{L}_\nu(t)}{t^\nu}\,\mathrm{d}t&>\frac{1}{1-\gamma}\bigg(\mathrm{e}^{-\gamma x}\frac{\mathbf{L}_{\nu}(x)}{x^\nu}-\frac{(1+\gamma x)(1-\mathrm{e}^{-\gamma x})}{\sqrt{\pi}\gamma2^\nu\Gamma(\nu+\frac{3}{2})}\bigg), \quad \nu>-\tfrac{3}{2}.
\end{align}
We have equality in (\ref{bi2}) and (\ref{bi3}) if $\nu=-\frac{1}{2}(n+1)$.  Inequalities (\ref{bi1})--(\ref{bi5}) are tight as $x\rightarrow\infty$ and inequality (\ref{bi3}) is also tight as $x\downarrow0$.   

Now suppose that $\nu>-\frac{1}{2}(n+1)$, and let 
\begin{equation*}D_{\nu,n}:=\sup_{x>0}\frac{x^\nu}{\mathbf{L}_{\nu+n}(x)}\int_0^x \frac{\mathbf{L}_{\nu+n}(t)}{t^\nu}\,\mathrm{d}t.
\end{equation*}
The existence of $D_{\nu,n}$ is guaranteed by inequalities (\ref{bi3}) and (\ref{Imon}), and we have $D_{\nu,n}<2(\nu+n+1)$.  Suppose also that $0<\gamma<\frac{1}{D_{\nu,n}}$.  Then, for all $x>0$,
\begin{align}\label{bi7}\int_0^x \mathrm{e}^{-\gamma t}\frac{\mathbf{L}_{\nu+n}(t)}{t^\nu}\,\mathrm{d}t&<\frac{\mathrm{e}^{-\gamma x}}{1-D_{\nu,n}\gamma}\int_0^x\frac{\mathbf{L}_{\nu+n}(t)}{t^\nu}\,\mathrm{d}t, \\
\int_0^x \mathrm{e}^{-\gamma t}\frac{\mathbf{L}_{\nu+n}(t)}{t^\nu}\,\mathrm{d}t&< \frac{\mathrm{e}^{-\gamma x}}{1-D_{\nu,n}\gamma}\bigg(\frac{2(\nu+n+1)}{n+1}\frac{\mathbf{L}_{\nu+n+1}(x)}{x^\nu} \nonumber \\
\label{bi8}&\quad-\frac{2\nu+n+1}{n+1}\frac{\mathbf{L}_{\nu+n+3}(x)}{x^\nu}+b_{\nu,n}x^{n+4}-c_{\nu,n}x^{n+2} \bigg).
\end{align}
\end{theorem}

\begin{proof}We first establish inequalities (\ref{bi1})--(\ref{bi8}) and then prove that the inequalities are tight in certain limits.

(i) From inequality (\ref{Imon}) we obtain
\begin{align*}\int_0^x \frac{\mathbf{L}_\nu(t)}{t^\nu}\,\mathrm{d}t>\int_0^x \frac{\mathbf{L}_{\nu+1}(t)}{t^\nu}\,\mathrm{d}t =\frac{\mathbf{L}_\nu(x)}{x^\nu}-\frac{x}{\sqrt{\pi}2^\nu\Gamma(\nu+\frac{3}{2})},
\end{align*}
where we used the differentiation formula (\ref{diffone}) and limiting form (\ref{Itend0}) to evaluate the integral.

(ii) The assertion that there is equality in (\ref{bi2}) and (\ref{bi3}) if $\nu=-\frac{1}{2}(n+1)$ can be seen from the fact the both these upper and lower bounds (which we now prove) are equal in this case.  We now suppose that $\nu>-\frac{1}{2}(n+1)$.  Consider the function
\begin{equation*}u(x)=\int_0^x \frac{\mathbf{L}_{\nu +n}(t)}{t^\nu}\,\mathrm{d}t-\frac{\mathbf{L}_{\nu+n+1}(x)}{x^\nu}+a_{\nu,n}x^{n+2}.
\end{equation*}
We argue that $u(x)>0$ for all $x>0$, which will prove the result.  We first note that from the differentiation formula (\ref{diffone}) followed by identity (\ref{Iidentity}) we have that
\begin{align}&\frac{\mathrm{d}}{\mathrm{d}x}\bigg(\frac{\mathbf{L}_{\nu+n+1}(x)}{x^\nu}\bigg)=\frac{\mathrm{d}}{\mathrm{d}x}\bigg(x^{n+1}\cdot\frac{\mathbf{L}_{\nu+n+1}(x)}{x^{\nu+n+1}}\bigg)\nonumber\\
&\quad=(n+1)\frac{\mathbf{L}_{\nu+n+1}(x)}{x^{\nu+1}}+\frac{\mathbf{L}_{\nu+n+2}(x)}{x^{\nu}}+\frac{x^{n+1}}{\sqrt{\pi}2^{\nu+n+1}\Gamma(\nu+n+\frac{5}{2})}\nonumber\\
&\quad=\frac{n+1}{2(\nu+n+1)}\bigg(\frac{\mathbf{L}_{\nu+n}(x)}{x^\nu}-\frac{\mathbf{L}_{\nu+n+2}(x)}{x^\nu}-\frac{x^{n+1}}{\sqrt{\pi}2^{\nu+n+1}\Gamma(\nu+n+\frac{5}{2})}\bigg)\nonumber\\
&\quad\quad+\frac{\mathbf{L}_{\nu+n+2}(x)}{x^\nu}+\frac{x^{n+1}}{\sqrt{\pi}2^{\nu+n+1}\Gamma(\nu+n+\frac{5}{2})}\nonumber \\
\label{num3}&\quad=\frac{n+1}{2(\nu+n+1)}\frac{\mathbf{L}_{\nu+n}(x)}{x^\nu}+\frac{2\nu+n+1}{2(\nu+n+1)}\frac{\mathbf{L}_{\nu+n+2}(x)}{x^\nu}+(n+2)a_{\nu,n}x^{n+1}.
\end{align}
Therefore
\begin{align*}u'(x)=\frac{2\nu+n+1}{2(\nu+n+1)}\bigg(\frac{\mathbf{L}_{\nu+n}(x)}{x^\nu}-\frac{\mathbf{L}_{\nu+n+2}(x)}{x^\nu}\bigg) >0,
\end{align*}
where we used (\ref{Imon}) to obtain the inequality.  Also, from (\ref{Itend0}), as $x\downarrow0$,
\begin{align*}u(x)&\sim \int_0^x \frac{t^{n+1}}{\sqrt{\pi}2^{\nu+n}\Gamma(\nu+n+\frac{3}{2})}\,\mathrm{d}t-\frac{x^{n+2}}{\sqrt{\pi}2^{\nu+n+1}\Gamma(\nu+n+\frac{5}{2})}+a_{\nu,n}x^{n+2}\\
&=\frac{x^{n+2}}{\sqrt{\pi}2^{\nu+n}(n+2)\Gamma(n+\nu+\frac{3}{2})} -\frac{x^{n+2}}{\sqrt{\pi}2^{\nu+n+1}\Gamma(\nu+n+\frac{5}{2})}+a_{\nu,n}x^{n+2}\\ 
&=\frac{x^{n+2}}{\sqrt{\pi}2^{\nu+n}\Gamma(\nu+n+\frac{3}{2})}\bigg(\frac{1}{n+2}-\frac{1}{2(\nu+n+\frac{3}{2})}\bigg)+a_{\nu,n}x^{n+2} >0,
\end{align*}
where the inequality can be seen to hold because $\nu>-\frac{1}{2}(n+1)$.  Thus, we conclude that $u(x)>0$ for all $x>0$, as required.

(iii) Integrating both sides of (\ref{num3}) over $(0,x)$, applying the fundamental theorem of calculus and rearranging gives
\begin{align*}\int_0^x \frac{\mathbf{L}_{\nu +n} (t)}{t^\nu}\,\mathrm{d}t &= \frac{2(\nu+n+1)}{n+1} \frac{\mathbf{L}_{\nu +n+1} (x)}{x^\nu} -\frac{2\nu+n+1}{n+1} \int_0^x \frac{\mathbf{L}_{\nu +n +2} (t)}{t^\nu}\,\mathrm{d}t \\
&\quad-\frac{2\nu+n+1}{n+1}\int_0^x\frac{t^{n+1}}{\sqrt{\pi}2^{\nu+n+1}\Gamma(\nu+n+\frac{5}{2})}\,\mathrm{d}t.
\end{align*}
Evaluating the second integral on the right hand-side of the above expression and using inequality (\ref{bi2}) to bound the first integral then yields (\ref{bi3}).

(iv) Let $\nu>-1$.   Then integration by parts and inequality (\ref{bi1}) gives
\begin{align*} \int_0^x \mathrm{e}^{-\gamma t} \frac{\mathbf{L}_\nu(t)}{t^\nu} \,\mathrm{d}t &= \mathrm{e}^{-\gamma x}\int_0^x \frac{\mathbf{L}_\nu(t)}{t^\nu}\,\mathrm{d}t + \gamma \int_0^x \mathrm{e}^{-\gamma t}\bigg(\int_0^t  \frac{\mathbf{L}_\nu(u)}{u^\nu} \,\mathrm{d}u\bigg) \,\mathrm{d}t \\
&> \mathrm{e}^{-\gamma x}\int_0^x \frac{\mathbf{L}_\nu(t)}{t^\nu}\,\mathrm{d}t + \gamma \int_0^x \mathrm{e}^{-\gamma t} \frac{\mathbf{L}_\nu(t)}{t^\nu}  \,\mathrm{d}t-\gamma \int_0^x \frac{t\mathrm{e}^{-\gamma t}}{\sqrt{\pi}2^\nu\Gamma(\nu+\frac{3}{2})} \,\mathrm{d}t,
\end{align*}
whence on evaluating $\int_0^xt\mathrm{e}^{-\gamma t} \,\mathrm{d}t=\frac{1}{\gamma^2}(1-(1+\gamma x)\mathrm{e}^{-\gamma x})$ and rearranging we obtain (\ref{bi3}).

(v) Apply inequality (\ref{bi1}) to inequality (\ref{bi4}).

(vi) We now prove inequality (\ref{bi7}); the assertion that $D_{\nu,n}<2(\nu+n+1)$ is immediate from inequalities (\ref{bi3}) and (\ref{Imon}).  Now, integrating by parts similarly to we did in part (iv), we have
\begin{align*}\int_0^x \mathrm{e}^{-\gamma t} \frac{\mathbf{L}_{\nu+n}(t)}{t^\nu} \,\mathrm{d}t &= \mathrm{e}^{-\gamma x}\int_0^x \frac{\mathbf{L}_{\nu+n}(t)}{t^\nu}\,\mathrm{d}t + \gamma \int_0^x \mathrm{e}^{-\gamma t}\bigg(\int_0^t  \frac{\mathbf{L}_{\nu+n}(u)}{u^\nu} \,\mathrm{d}u\bigg) \,\mathrm{d}t \\
&<\mathrm{e}^{-\gamma x}\int_0^x \frac{\mathbf{L}_{\nu+n}(t)}{t^\nu}\,\mathrm{d}t + D_{\nu,n}\gamma \int_0^x \mathrm{e}^{-\gamma t} \frac{\mathbf{L}_{\nu+n}(t)}{t^\nu}  \,\mathrm{d}t.
\end{align*}
As we assumed $0<\gamma<\frac{1}{D_{\nu,n}}$, on rearranging we obtain inequality (\ref{bi7}).

(vii) Apply inequality (\ref{bi3}) to inequality (\ref{bi7}).

(viii) Finally, we prove that inequalities (\ref{bi1})--(\ref{bi5}) are tight as $x\rightarrow\infty$ and inequality (\ref{bi3}) is also tight as $x\downarrow0$.  We begin by noting that a straightforward asymptotic analysis using (\ref{Itendinfinity}) gives that, for $0\leq\gamma<1$, $n>-\tfrac{3}{2}$ and $\nu\in\mathbb{R}$, 
\begin{equation}\label{eqeq1} \int_0^x \mathrm{e}^{-\gamma t}\frac{\mathbf{L}_{\nu+n}(t)}{t^\nu}\,\mathrm{d}t\sim \frac{1}{\sqrt{2\pi}(1-\gamma)}x^{-\nu-1/2}\mathrm{e}^{(1-\gamma)x}, \quad x\rightarrow\infty,
\end{equation}
and we also have
\begin{equation}\label{eqeq2}\mathrm{e}^{-\gamma x}\frac{\mathbf{L}_{\nu+n}(x)}{x^\nu}\sim  \frac{1}{\sqrt{2\pi}}x^{-\nu-1/2}\mathrm{e}^{(1-\gamma)x}, \quad x\rightarrow\infty.
\end{equation}
One can now readily check with the aid of (\ref{eqeq1}) and (\ref{eqeq2}) that inequalities (\ref{bi1})--(\ref{bi5}) are tight as $x\rightarrow\infty$.  

It now remains to prove that inequality (\ref{bi3}) is tight as $x\downarrow0$.  From (\ref{Itend0}), we have on the one hand, as $x\downarrow0$,
\begin{equation*}\int_0^x \frac{\mathbf{L}_{\nu+n}(t)}{t^\nu}\,\mathrm{d}t\sim\int_0^x \frac{t^{n+1}}{\sqrt{\pi}2^{\nu+n}\Gamma(\nu+n+\frac{3}{2})}\,\mathrm{d}t=\frac{x^{n+2}}{\sqrt{\pi}2^{\nu+n}(n+2)\Gamma(\nu+n+\frac{3}{2})},
\end{equation*}
and on the other,
\begin{align*}&\frac{2(\nu+n+1)}{n+1}\frac{\mathbf{L}_{\nu+n+1}(x)}{x^\nu}-\frac{2\nu+n+1}{n+1}\frac{\mathbf{L}_{\nu+n+3}(x)}{x^\nu}+b_{\nu,n}x^{n+2}+c_{\nu,n}x^{n+4} \\
&\quad \sim \frac{(\nu+n+1)x^{n+2}}{\sqrt{\pi}2^{\nu+n}(n+1)\Gamma(\nu+n+\frac{5}{2})}-\frac{(2\nu+n+1)x^{n+2}}{\sqrt{\pi}2^{\nu+n+1}(n+1)(n+2)\Gamma(\nu+n+\frac{5}{2})} \\
&\quad=\frac{(2(\nu+n+1)(n+2)-(2\nu+n+1))x^{n+2}}{\sqrt{\pi}2^{\nu+n+1}(n+1)(n+2)\Gamma(\nu+n+\frac{5}{2})} \\
&\quad=\frac{2(n+1)(\nu+n+\frac{3}{2})x^{n+2}}{\sqrt{\pi}2^{\nu+n+1}(n+1)(n+2)\Gamma(\nu+n+\frac{5}{2})} =\frac{x^{n+2}}{\sqrt{\pi}2^{\nu+n}(n+2)\Gamma(\nu+n+\frac{3}{2})},
\end{align*}
where we used that $u\Gamma(u)=\Gamma(u+1)$.  This proves the claim.
\end{proof}

\begin{remark}The constants $D_{\nu,n}$ can be computed numerically.  As an example, we used \emph{Mathematica} to find $D_{0,0}=1.109$, $D_{1,0}=1.331$ $D_{3,0}=1.693$, $D_{5,0}=1.990$ and $D_{10,0}=2.584$.
\end{remark}


\begin{remark}The upper bounds (\ref{bi7}) and (\ref{bi8}) are not tight in the limits $x\downarrow0$ and $x\rightarrow\infty$, but they are of the correct order in both limits ($O(x^{n+1})$ as $x\downarrow0$, and $O(x^{-\nu-1/2}\mathrm{e}^{(1-\gamma)x})$ as $x\rightarrow\infty$).  The bounds are simple but are not entirely satisfactory in that they only hold for $0<\gamma<\frac{1}{D_{\nu,n}}$, whereas one would like the inequalities to be valid for all $0<\gamma<1$.  It should be mentioned that a similar problem was encountered by \cite{gaunt ineq3} in that the upper bounds obtained for $\int_0^x \mathrm{e}^{-\gamma t} t^\nu I_\nu(t)\,\mathrm{d}t$ were only valid for $0<\gamma<\alpha_\nu$, for some $0<\alpha_\nu<1$.  
\end{remark}


We end by noting that one can combine the inequalities of Theorem \ref{tiger1} and the integral formula (\ref{besint6}) to obtain lower and upper bounds for a generalized hypergeometric function.  We give an example in the following corollary. 


\begin{corollary}\label{struvebessel}Let $\nu>\frac{1}{2}$. Then, for all $x>0$,
\begin{align}\mathbf{L}_{\nu}(x)-a_{\nu-1,0}x^{\nu+1}&<\frac{x^{\nu+1}}{\sqrt{\pi}2^{\nu}\Gamma(\nu+\frac{1}{2})}{}_2F_3\bigg(1,1;\frac{3}{2},2,\nu+\frac{1}{2};\frac{x^2}{4}\bigg)\nonumber \\
\label{dob22}&<2\nu \mathbf{L}_\nu(x)-(2\nu-1)\mathbf{L}_{\nu+2}(x)+b_{\nu-1,0}x^{\nu+3}-c_{\nu-1,0}x^{\nu+1}.
\end{align}
\end{corollary}

\begin{proof}Combine the integral formula (\ref{besint6}) and inequalities (\ref{bi2}) and (\ref{bi3}) (with $n=0$) of Theorem \ref{tiger1}, and replace $\nu$ by $\nu-1$. 
\end{proof}

\begin{remark}We know from Theorem \ref{tiger1} that the two-sided inequality (\ref{dob22}) is tight in the limit $x\rightarrow\infty$, and the upper bound is also tight as $x\downarrow0$.  To elaborate further, we denote by $F_\nu(x)$ the expression involving the generalized hypergeometric function in (\ref{dob22}), and the lower and upper bounds by $L_\nu(x)$ and $U_\nu(x)$.  We used \emph{Mathematica} to compute the relative error in approximating $F_\nu(x)$ by $L_\nu(x)$ and $U_\nu(x)$, and numerical results are given in Tables \ref{table1} and \ref{table2}.  We observe that, for a given $x$, the relative error in approximating $F_\nu(x)$ by either $L_\nu(x)$ or $U_\nu(x)$ increases as $\nu$ increases.  We also notice from Table \ref{table1} that, for a given $\nu$, the relative error in approximating $F_\nu(x)$ by $L_\nu(x)$ decreases as $x$ increases.  However, from Table \ref{table2} we see that, for a given $\nu$, as $x$ increases the relative error in approximating $F_\nu(x)$ by $U_\nu(x)$ initially increases before decreasing.  This is because the upper bound is tight as $x\downarrow0$. 

\begin{table}[h]
\centering
\caption{\footnotesize{Relative error in approximating $F_\nu(x)$ by $L_\nu(x)$.}}
\label{table1}
{\scriptsize
\begin{tabular}{|c|rrrrrrr|}
\hline
 \backslashbox{$\nu$}{$x$}      &    0.5 &    5 &    10 &    25 &    50 &    100 & 250   \\
 \hline
1 & 0.4959 & 0.2540 & 0.1089 & 0.0409 & 0.0202 & 0.0101 & 0.0040  \\
2.5 & 0.7979 & 0.6225 & 0.3708 & 0.1539 & 0.0784 & 0.0396 & 0.0159  \\
5 & 0.8992 & 0.8229 & 0.6374 & 0.3130 & 0.1678  & 0.0869 & 0.0355 \\
7.5 & 0.9329 & 0.8923 & 0.7741 & 0.4407 & 0.2482 & 0.1318 & 0.0547  \\ 
10 & 0.9498 & 0.9249 & 0.8472  & 0.5426  &  0.3205   & 0.1745 & 0.0735 \\  
  \hline
\end{tabular}
}
\end{table}
\begin{table}[h]
\centering
\caption{\footnotesize{Relative error in approximating $F_\nu(x)$ by $U_\nu(x)$.}}
\label{table2}
{\scriptsize
\begin{tabular}{|c|rrrrrrr|}
\hline
 \backslashbox{$\nu$}{$x$}      &    0.5 &    5 &    10 &    25 &    50 &    100 & 250   \\
 \hline
1 & 0.0041 & 0.1939 & 0.1981 & 0.1034 & 0.0558 & 0.0289 & 0.0118  \\
2.5 & 0.0070 & 0.5184 & 0.9270 & 0.6847 & 0.4073 & 0.2213 & 0.0930  \\
5 & 0.0062 & 0.5679 & 1.6268 & 2.0626 & 1.4411  & 0.8462 & 0.3721 \\
7.5 & 0.0051 & 0.4985 & 1.7368 & 3.4231 & 2.7983 & 1.7750 & 0.8169  \\ 
10 & 0.0043 & 0.4285 & 1.6301  & 4.5028  & 4.2818    & 2.9312 & 1.3959 \\  
  \hline
\end{tabular}
}
\end{table}
\end{remark}

\subsection*{Acknowledgements}
The author is supported by a Dame Kathleen Ollerenshaw Research Fellowship.

\end{document}